\documentclass[11pt,a4paper]{article}
\usepackage[utf8]{inputenc}
\usepackage[english]{babel}
\usepackage[T1]{fontenc}
\usepackage[pdftex]{hyperref}
\usepackage{comment}
\hypersetup{
    colorlinks=true,
    linkcolor=blue,
    filecolor=magenta,    
    urlcolor=cyan,
    citecolor=blue
}
\usepackage{amsmath,amsthm,amssymb,enumerate,graphicx,mathtools,mathrsfs}
\usepackage{algorithm,algpseudocode,float, caption, subcaption}
\usepackage{epstopdf}

\algnewcommand\algorithmicforeach{\textbf{for each}}
\algdef{S}[FOR]{ForEach}[1]{\algorithmicforeach\ #1\ \algorithmicdo}

\newlength{\continueindent}
\usepackage{etoolbox}
\makeatletter
\newcommand*{\ALG@customparshape}{\parshape 2 \leftmargin \linewidth \dimexpr\ALG@tlm+\continueindent\relax \dimexpr\linewidth+\leftmargin-\ALG@tlm-\continueindent\relax}
\apptocmd{\ALG@beginblock}{\ALG@customparshape}{}{\errmessage{failed to patch}}
\makeatother

\newcommand{\Pois}[1]{\text{Pois}\left(#1\right)}

\newcommand{\G}{\mathcal{G}}

\newcommand{\ex}{\mathbb{E}}
\newcommand{\PP}{\mathbb{P}}

\newcommand{\bx}{\mathbf{x}}

\DeclarePairedDelimiter{\abs}{\lvert}{\rvert}

\newtheorem{theorem}{Theorem}
\newtheorem{corollary}[theorem]{Corollary}

\newtheorem{remark}[theorem]{Remark}
\newtheorem{proposition}[theorem]{Proposition}
\newtheorem{lemma}[theorem]{Lemma}

\newcommand{\mB}{\mathsf{B}}

\newcommand{\mG}{\mathsf{G}}

\newcommand{\mT}{\mathsf{T}}

\title{Limits of chordal graphs with bounded tree-width}
\author{Jordi Castellví\thanks{Centre de Recerca Matemàtica, Barcelona. E-mail: jcastellvi@crm.cat} \and Benedikt Stufler\thanks{Vienna University of Technology. E-mail: benedikt.stufler@tuwien.ac.at}}

\begin{document}

\maketitle

\begin{abstract}
We study random $k$-connected chordal graphs with bounded tree-width. Our main results are scaling limits and quenched local limits.
\\


\textbf{MSC classes:} 60C05, 05C80.
\end{abstract}

\section{Introduction}

A graph is \emph{chordal} if every cycle of length at least $4$ contains a chord, i.e., an 
edge between two non-adjacent vertices of the cycle. Chordal graphs can be 
characterized as the graphs in which all minimal separating sets are cliques.
A graph is \emph{$k$-connected} if if has at least $k$ vertices and cannot be
disconnected by removing fewer than $k$ vertices. The graphon limit of
chordal graphs was established by Janson~\cite{JANSON2016}. The
structural and enumerative study of chordal graphs was pioneered by Wormald 
\cite{WORMALD1985}, where he developed a method to find the exact number of chordal
graphs with exactly $n$ vertices for a given $n$ using generating functions. As he also
noted, an interesting property of chordal graphs is that they admit a decomposition into
$k$-connected components for any $k\in\mathbb{N}$, which is not the case for arbitrary 
graphs, where this is only true until $k=3$ (for a detailed explanation of this see 
\cite{CDNR2023}).

\emph{Tree-width} is a fundamental parameter in structural and algorithmic graph theory 
which can be defined in terms of tree-decompositions or, equivalently, in terms of 
$k$-trees. A $k$-tree is a graph obtained from a $(k+1)$-clique by iteratively connecting a 
new vertex to all the vertices of an already existing $k$-clique. Note that $1$-trees are just 
trees with at least two vertices.  The tree-width of a graph $\Gamma$ is the minimum $k$ 
such that $\Gamma$ is a subgraph of a $k$-tree.  In particular, $k$-trees are the 
edge-maximal graphs with tree-width at most $k$. The graph limits of $k$-trees have
been studied both in the labelled and unlabelled setting, see \cite{zbMATH07118076} and 
\cite{JS2020}, respectively.

Labelled chordal graphs with tree-width at most $t$ have been recently enumerated in
\cite{CDNR2023} by means of their decomposition into $k$-connected components for 
$k=1, \dots, t+1$. The only $(t+1)$-connected (chordal) graph with tree-width at most $t$ 
is the $(t+1)$-clique.

For fixed $n, t>0$ and $0\leq k\leq t$, let $\G_{t, k, n}$ denote the set of $k$-connected chordal
graphs with $n$ labelled vertices and tree-width at most $t$. Then, by the first main result of \cite{CDNR2023} there exist  constants $c_{t,k} > 0$ and
$\rho_{t,k}\in (0,1)$ such that
\begin{equation}
	\label{eq:asymnum}
    |\mathcal{G}_{t,k,n}| = c_{t,k}\, n^{-5/2}\, \rho_{t,k}^{-n}\, n! \, (1 + o(1))
    \qquad\text{as } n\to\infty.
\end{equation}
Furthermore, the second result of \cite{CDNR2023} establishes a multi-dimensional central limit theorem for the number of cliques in the uniformly at random selected graph $\mG_{t,k,n}$ from~$\G_{t,k, n}$.
Let $X_{n,i}$ denote the number of $i$-cliques in $\mG_{t,k,n}$ and set ${\bf X_n} = (X_{n,2},\ldots, X_{n,t})$.
Then, we have that
\begin{equation}
	\label{eq:cliqueclt}
    \frac{1}{\sqrt n} \left( {\bf X}_n - \mathbb{E}\, {\bf X}_n\right) \overset{d}\to N(0,{\boldsymbol \Sigma}),
    \quad\text{with }  \mathbb{E}\, {\bf X}_n \sim {\boldsymbol \alpha} n \text{ and } \mathbb{C}\mathrm{ov}\, {\bf X}_n \sim {\boldsymbol \Sigma} n,
\end{equation}
where ${\boldsymbol \alpha} = (\boldsymbol \alpha_i)_{1 \le i \le t-1}$ is a $(t-1)$-dimensional vector of positive numbers and ${\boldsymbol \Sigma}$ is a $(t-1)\times(t-1)$-dimensional positive semi-definite matrix. 

In this paper, we extend the analytic study of labelled $k$-connected chordal graphs with
bounded tree-width by introducing probabilistic methods. Instead of analyzing individual
parameters using systems of equations for generating functions, our approach is to 
consider random graphs with a given number of vertices chosen uniformly. We then 
establish limit objects of the sequence of random graphs that encode asymptotic 
properties of the finite model. This is facilitated via hidden multi-type branching processes. In what follows, chordal graphs with bounded tree-width
will be labelled unless otherwise specified.

Our first result establishes the Brownian tree $\left( \mathscr{T}_e, d_{\mathscr{T}_e} , \mu_{\mathscr{T}_e}
\right)$ constructed in~\cite{MR1085326,MR1166406,MR1207226} as the Gromov--Hausdorff--Prokhorov scaling limit.

\begin{theorem}
	\label{te:scaling}
    There is a constant $\kappa_{t, k}>0$ such that
    \[
        \left( \mG_{t,k,n}, \kappa_{t, k}n^{-1/2}d_{\mG_{t,k,n}}, \mu_{\mG_{t,k,n}} \right) \stackrel{d}{\longrightarrow}
        \left( \mathscr{T}_e, d_{\mathscr{T}_e}, \mu_{\mathscr{T}_e} \right)
    \]
    in the Gromov--Hausdorff--Prokhorov sense as $n \to \infty$.
\end{theorem}

Here $d_{\mG_{t,k,n}}$ refers to the graph distance on the vertex set of $\mG_{t,k,n}$, and $\mu_{\mG_{t,k,n}}$ to the uniform measure on this set. Notice that when $k=t$ we recover the result for $t$-trees \cite[Thm. 1]{DJS2016} and that when $k=1$ we recover the result for the connected family, which is mentioned in \cite{CDNR2023} as a consequence of a general theorem about subcritical families \cite{PSW2016}.

Our result contributes to the universality of the Brownian tree, referring to the phenomenon that it arises as Gromov--Hausdorff--Prokhorov scaling limit of several  different models~\cite{MR3382675,MR3342658,MR3573291,MR4132643,bienvenu2022branching,manus,MR3853863}.

Apart from describing the asymptotic global shape, we prove a quenched limit that determines the asymptotic local shape:

\begin{theorem}
	\label{te:local}
    Let $v_n$ denote a uniformly selected vertex of $\mG_{t,k,n}$.  There exists an infinite vertex-rooted chordal graph $\hat{\mG}_{t,k}$ such that
    \[
    	(\mG_{t,k,n}, v_n) \stackrel{d}{\longrightarrow}  \hat{\mG}_{t,k}
    \]
    in the local topology as $n \to \infty$. Furthermore, we have quenched convergence
    \[
    	\mathfrak{L}((\mG_{t,k,n}, v_n) \mid \mG_{t,k,n} )  \stackrel{p}{\longrightarrow} \mathfrak{L}(\hat{\mG}_{t,k}).
    \]
\end{theorem}

Our third result is a tail-bound for the diameter  $\mathrm{D}(\mG_{t,k,n})$ of $\mG_{t,k,n}$. Such bounds have been determined for models of random trees~\cite{MR3077536,MR4479914,addarioberry2022random}, and in fact our proof makes use of the main result of~\cite{addarioberry2022random}. 

\begin{theorem}
	\label{te:tailbound}
	There exist constants $C,c>0$ (that may depend on $t$ and $k$) such that for all $n \ge 1$ and $x>0$
	\[
		\mathbb{P}(\mathrm{D}(\mG_{t,k,n}) \ge x) \le C \exp(-cx^2 / n).
	\]
\end{theorem}

An application of these tail-bounds is that, by a standard calculation, they imply that the rescaled diameter $\mathrm{D}(\mG_{t,k,n}) / \sqrt{n}$ is $p$-uniformly integrable for any fixed $p\ge1$. Theorem~\ref{te:scaling} entails distributional convergence
\[
	\mathrm{D}(\mG_{t,k,n}) \kappa_{t, k}n^{-1/2} \stackrel{d}{\longrightarrow} \mathrm{D}(\mathscr{T}_e).
\]
Arbitrarily large uniform integrability allows us to apply the mean convergence criterion, yielding
\[
	\mathbb{E}[\mathrm{D}(\mG_{t,k,n})^p] \kappa_{t, k}^p n^{-p/2} \to \mathbb{E}[\mathrm{D}(\mathscr{T}_e)^p].
\]
The distribution and  moments of the diameter of the Brownian tree are  known~\cite{MR3434205,MR0731595,MR1166406}.

In a similar fashion our results entail distributional convergence and convergence of moments for other Gromov--Hausdorff--Prokhorov continuous functionals that are bounded by the diameter, for example the eccentricity with respect to a random vertex, or the distance between two or more uniformly selected vertices.

\subsection*{Notation}

We recall some notation that will be used in the following sections. The positive integers are denoted by $\mathbb{N}$, and the non-negative integers by $\mathbb{N}_0$. We denote by $\mathfrak{L}(X)$ the law of a random variable $X$.

All unspecified limits are as $n\to\infty$. We use $\stackrel{d}{\longrightarrow}$ and 
$\stackrel{p}{\longrightarrow}$  for convergence in distribution and probability of
random variables. We say that some event happens with high probability if its probability
tends to $1$. We let $\mathcal{O}_p(a_n)$ denote the product of a real number $a_n$ with a stochastically bounded random variable $Z_n$ that we do not specify explicitly. We let $d_{\mathrm{TV}}$ denote the total variation distance between (the laws of) random variables.

\section{Background}

\subsection{Local convergence}

Let $\mathfrak{G}$ denote the collection of rooted locally finite unlabelled graphs. We
define the local distance between two graphs $G, H\in \mathfrak{G}$ as
\[
    d_{\text{loc}}(G, H) = 2^{-\sup\{m\in\mathbb{N}_0 \,\mid\, U_m(G) = U_m(H)\}},
\]
where $U_m(G)$ is the $m$-neighbourhood of the root of $G$, that is, the graph induced 
by all vertices in $G$ at distance at most $m$ from its root. $d_{\text{loc}}$ is indeed a
metric and $(\mathfrak{G}, d_{\text{loc}})$ is a Polish space \cite[Theorem 1]{CURIEN2018},
so we are in a standard setting for studying distributional convergence of random elements of this space. If $X_{\infty}, X_1, X_2, \ldots$ denote  random
graphs from $\mathfrak{G}$, then 
\[
    X_n\stackrel{d}{\longrightarrow} X_{\infty}
\]
if for any bounded continuous function $F: \mathfrak{G}\to\mathbb{R}$ we have that
\[
    \ex [F(X_n)] \to \ex[F(X_{\infty})].
\]
This is in fact equivalent to the condition that for all $r\in\mathbb{N}_0$ and all graphs
$G\in \mathfrak{G}$ it holds that
\[
    \PP(U_r(X_n)=G) \to \PP(U_r(X_{\infty})=G).
\]

\subsection{Global convergence}

The Gromov--Hausdorff--Prokhorov  distance $d_{\mathrm{GHP}}$ is a well-known metric on the collection $\mathfrak{K}$ of  isometry-equivalence classes of compact metric spaces endowed with Borel probability measures.   Detailed expositions of this concept can be found in \cite{zbMATH06610054}, \cite[Ch. 7]{MR1835418},~\cite[Ch. 27]{zbMATH05306371},~\cite{zbMATH06247183}, \cite[Sec. 6]{MR2571957}, and~\cite{janson2020gromovprohorov}.

\section{Sampling Procedures}

\subsection{Generating functions of labelled chordal graphs with bounded tree-width}
For the remainder of this paper, we fix an integer $t\ge 1$ and  omit the index~  $t$ from notation, so that the dependency on $t$ is implicit. In accordance with this convention, we let $\mathcal{G}$ denote the class of
chordal graphs with tree-width at most $t$. 

For a graph $\Gamma \in \mathcal{G}$ and 
$j\in[t]$, let us denote by $n_j(\Gamma)$ the number of $j$-cliques of $\Gamma$.
Define the multivariate (exponential) generating function associated to $\mathcal{G}$ to be
\begin{equation*}
    G(\bx) = G(x_1, \dots, x_t) = \sum_{\Gamma \in \mathcal{G}} \frac{1}{n_1(\Gamma)!}
    \prod_{j=1}^{t} x_j^{n_j(\Gamma)},
\end{equation*}
Let $g_n$ denote the number of chordal graphs with $n$  vertices and tree-width at most 
$t$. 
Then,
\begin{equation*}
    G(x, 1, \dots, 1) = \sum_{n\geq 1} \frac{g_n}{n!} x^n.
\end{equation*}
For $0\le k\le t+1$, let $\mathcal{G}_k$ be the class of $k$-connected chordal graphs
with tree-width at most $t$ and $G_k(\bx)$ be the associated generating function.
In particular, for $k=t+1$ the only graph in the class is the $(t+1)$-clique:
\begin{equation*}
    G_{t+1}(\bx) = \frac{1}{(t+1)!}\prod_{j\in[t]}x_j^{\binom{t+1}{j}}.
\end{equation*}

Rooting the graph $\Gamma\in \mathcal{G}_k$ at an (ordered)
$i$-clique means distinguishing one $i$-clique $K$ of $\Gamma$, choosing an
ordering of  its  vertices and forgetting their labels.
In order to avoid over-counting, we will discount the subcliques of $K$.
Let $i\in [k]$ and define $\mathcal{G}_{k}^{(i)}$ to be the class of $k$-connected chordal 
graphs with tree-width at most $t$ and rooted at an $i$-clique.
Let then $G_k^{(i)}(\bx)$ be the associated generating function, where now for $1\leq j \leq 
i$ the variables $x_j$ mark the number of $j$-cliques that are not subcliques of the root.
Then, for $k \in[t]$, the following equations hold \cite[Lem. 2.5]{CDNR2023}:
\begin{align}
    G_{k+1}^{(k)}(\bx) &= k!\left(\prod_{j=1}^{k-1}x_j^{-\binom{k}{j}}\right)\frac{\partial}{\partial x_k} 
    G_{k+1}(\bx), \label{eq:root}\\
    G_k^{(k)}(\bx) &= \exp \left( G_{k+1}^{(k)}\big( x_1, \ldots, x_{k-1}, x_kG_k^{(k)}(\bx), 
    x_{k+1}, \ldots, x_t \big) \right) , \label{eq:recursive}\\
    G_k(\bx) &= \frac{1}{k!}\left(\prod_{j=1}^{k-1}x_j^{\binom{k}{j}}\right)\int G_k^{(k)}(\bx)\,\, dx_k. \label{eq:unroot}
\end{align}
For ease of notation, we set $G_k^{(k)}(x) := G_k^{(k)}(x, 1, \dots, 1)$, so that
\[
G_k^{(k)}(x) =  \exp (G_{k+1}^{(k)}(x , 1, \dots, 1, G_k^{(k)}(x), 1, \dots, 1 ) ).
\]
Let $\rho_k$ be radius of convergence of $G_k^{(k)}(x)$. We recall the fact that Equation \eqref{eq:recursive} is subcritical  
\cite[Lem. 3.7]{CDNR2023}, implying that
\begin{equation}
	\label{eq:subcrit}
	G_{k+1}^{(k)}\big( \rho_k + \epsilon, 1 \ldots, 1, G_k^{(k)}(\rho_k) + \epsilon, 1, \ldots, 1 \big) < \infty
\end{equation}
for some $\epsilon>0$. Moreover, we can get the asymptotics of $\abs{\mathcal{G}_{k, n}^{(k)}}$ as follows. For ease of notation, we set $G_k(x_1, x_k) := G_k(x_1, 1, \dots, 1, x_k, 1, \dots, 1)$ and $G_k^{(k)}(x_1, x_k) := G_k^{(k)}(x_1, 1, \dots, 1, x_k, 1, \dots, 1)$.
By~\eqref{eq:root},
\begin{align*}
    \frac{\abs{\mathcal{G}_{k, n}^{(k)}}}{n!} &= [x_1^n] G_k^{(k)}(x_1, x_k) \\
    & =k![x_1^n] \frac{\partial}{\partial x_k} G_k(x_1, x_k) x_1^{-k} \\
    & =k![x_1^n] \sum_{\Gamma \in \mathcal{G}_k} n_k(\Gamma) \frac{x_1^{n_1(\Gamma)-k}}{n_1(\Gamma)!} \\
    & =k![x_1^n] \sum_{\Gamma \in \mathcal{G}_{k, n+k}} n_k(\Gamma) \frac{x_1^{n}}{(n+k)!}  \\
    & = \frac{k!}{(n+k)!} \sum_{\Gamma \in \mathcal{G}_{k, n+k}} n_k(\Gamma).
\end{align*}
Therefore, by \eqref{eq:cliqueclt},
\begin{equation} \label{eq:asymptoticclique}
    \binom{n+k}{k}\abs{\mathcal{G}_{k, n}^{(k)}} \sim n \boldsymbol \alpha_k \abs{\mathcal{G}_{k, n+k}}, \qquad \text{as } n\to\infty,
\end{equation}
where $\boldsymbol\alpha_k$ is the $k$-th component of $\boldsymbol\alpha$ .

\subsection{Boltzmann sampling procedure}

We describe a Boltzmann sampler $\Gamma G_k^{(k)}(\rho_k)$ that samples graphs in
$\mathcal{G}_{k}^{(k)}$ according to
\[
    \PP(\Gamma G_k^{(k)}(\rho_k) = A) = \frac{\rho_k^{\abs{A}}}{\abs{A}! G_k^{(k)}(\rho_k)}.
\]
In particular, $\Gamma G_k^{(k)}(\rho_k)$  conditioned on having $n$ vertices is uniformly 
distributed on $\mathcal{G}_{k, n}^{(k)}$.

In order to describe $\Gamma G_k^{(k)}(\rho_k)$, will use an auxiliary Boltzmann sampler $\Gamma S(\rho_k) :=
\Gamma  \exp (G_{k+1}^{(k)}(\rho_k , 1, \dots, 1, G_k^{(k)}(\rho_k), 1, \dots, 1 ) )$, for which we 
give no description. $ \Gamma S(\rho_k)$ samples a set of graphs in
$\mathcal{G}_{k+1}^{(k)}$ with distribution $\Pois{G_{k+1}^{(k)}(\rho_k , 1, \dots, 1, 
G_k^{(k)}(\rho_k), 1, \dots, 1 )}$.

\begin{algorithm}[H]
    \caption{Procedure for $\Gamma G_k^{(k)}(\rho_k)$}
    \label{alg_bfs}
    \begin{algorithmic}[1]
        \State let $R$ be a root $k$-clique with ordered vertices
        \State let $\text{result} := R$
        \State sample a set of graphs $D: = \Gamma S(\rho_k)$
        \ForEach {each graph $E$ in $D$}
            \ForEach {each $k$-clique $C$ in $E$ different from its root}
                \State sample $H:= \Gamma G_k^{(k)}(\rho_k)$
                \State attach $H$ to $E$ by identifying $C$ with the root of $H$ in the unique
                way such that the order of the root vertices respects the labels of $C$.
            \EndFor
            \State attach $E$ to $\text{result}$ by identifying $R$ with the root of $E$ in the
            unique way such that the order of the vertices coincides
        \EndFor
        \State distribute labels to all vertices not in $R$ uniformly at random
        \State \textbf{return} result
    \end{algorithmic}
\end{algorithm}

\subsection{Blow-up sampling procedure}

\label{sec:blowup}

We are going to reformulate the Boltzmann sampling algorithm so that  it first generates a random tree and then creates a graph from that tree using a blow-up procedure.

To this end, we define a pair $(\xi, \zeta)$ of random non-negative integers with probability generating function
\begin{align}
	\label{eq:offspring}
	\ex[z^\xi w^\zeta] = \exp( G_{k+1}^{(k)}(w \rho_k, 1, \ldots, 1, z G_k^{(k)}(\rho_k), 1, \ldots, 1) )  / G_k^{(k)}(\rho_k).
\end{align}
For ease of reference, if we consider a graph rooted at a $k$-clique, we refer to the vertices not contained in that clique as non-root vertices, and to the $k$-cliques different from the root $k$-clique also as non-root $k$-cliques. Thus, $\xi$ and $\zeta$ are distributed like the numbers of non-root $k$-cliques and vertices in $\Gamma S(\rho_k)$.

Since $\Gamma S(\rho_k)$ follows a Boltzmann distribution, conditioning it on a given number of vertices yields the uniform distribution among all possible outcomes with that number of vertices. Thus, if we perform a two-step procedure where we first generate $(\xi, \zeta)$ and then generate a $\mathrm{SET}(\mathcal{G}_{k+1}^{(k)})$ object with accordingly many non-root $k$-cliques and vertices, we create a random structure that is distributed like $\Gamma  S(\rho_k)$.

We let $\mT$ denote a $2$-type Bienaym\'e--Galton--Watson tree such that vertices of the second kind are infertile, and each vertex of the first kind receives mixed offspring according to an independent copy of $(\xi, \zeta)$. The root of $\mT$ is always defined to have type $1$. For ease of reference, we refer to vertices of the first type as `black', and vertices of the second type as `white'. We consider $\mT$ as ordered, so that there is a total order on the black children of any vertex, and a total order on the white children of any vertex. We do not impose any ordering between vertices of different types. Figure \ref{fig:tree} contains an example of this type of trees with $k=2$.

\begin{figure}[h]
    \centering
    \includegraphics{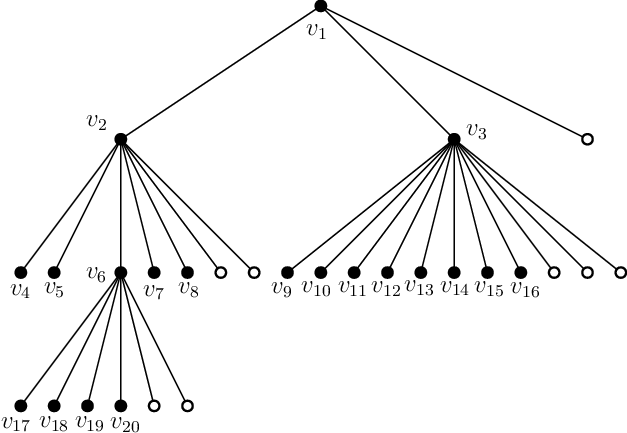}
    \caption{Example of a $2$-type Bienaym\'e--Galton--Watson tree.}
    \label{fig:tree}
\end{figure}

It follows from the recursive description of $\Gamma G_k^{(k)}(\rho_k)$ in Algorithm~\ref{alg_bfs} that the numbers of black and white vertices in $\mT$ are distributed like the  total number of $k$-cliques (including the root clique) and the  number of non-root vertices of $\Gamma G_k^{(k)}(\rho_k)$. In particular, since the graph $\Gamma G_k^{(k)}(\rho_k)$ is almost surely finite, it follows that  the tree $\mT$ is almost surely finite as well.

\begin{figure}[h]
    \centering
    \captionsetup[subfigure]{labelformat=empty}
    \begin{subfigure}{0.25\textwidth}
        \centering
        \includegraphics[width=\textwidth]{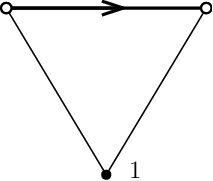}
        \caption{$\alpha(v_1)$}
    \end{subfigure}
    \hspace{3em}
    \begin{subfigure}{0.25\textwidth}
        \centering
        \includegraphics[width=\textwidth]{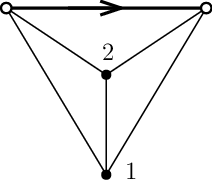}
        \caption{$\alpha(v_2)$}
    \end{subfigure}

    \vspace{1ex}

    \begin{subfigure}{0.25\textwidth}
        \centering
        \includegraphics[width=\textwidth]{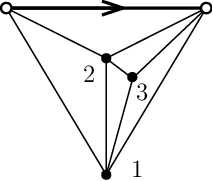}
        \caption{$\alpha(v_3)$}
    \end{subfigure}
    \hspace{3em}
    \begin{subfigure}{0.5\textwidth}
        \centering
        \includegraphics[width=\textwidth]{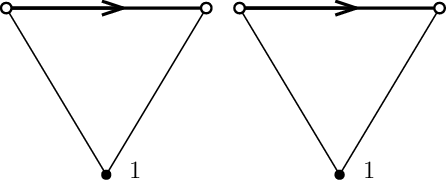}
        \caption{$\alpha(v_6)$}
    \end{subfigure}
    \caption{The non-empty decorations of the tree in Figure \ref{fig:tree}. The thick oriented edges are the ordered roots.}
    \label{fig:tree_decorations}
\end{figure}

Since we can generate $\Gamma S(\rho_k)$ from $(\xi, \zeta)$, we can also generate $\Gamma G_k^{(k)}(\rho_k)$ from $\mT$ by drawing for each black vertex $v$ of $\mT$ a uniform and independent \emph{decoration} $\alpha(v)$ from $\mathrm{SET}(\mathcal{G}_{k+1}^{(k)})$ such that the numbers of non-root $k$-cliques and vertices of $\alpha(v)$ agrees with the numbers of black and white children of $v$. Note that $\mathrm{SET}(\mathcal{G}_{k+1}^{(k)})$ contains a unique structure of size $0$, so this also makes sense when $v$ has no offspring.  The decoration $\alpha(v)$ may be viewed as a single connected graph by gluing together the root $k$-cliques of its components. See Figure \ref{fig:tree_decorations} for a possible decoration of the tree in Figure \ref{fig:tree}, with $k=2$. We refer to the pair $(\mT, \alpha)$ as a \emph{decorated tree}. The chordal graph corresponding to this decorated tree is obtained by gluing the decorations together according to the tree-structure $\mT$. That is, we start with the decoration of the root vertex $o$ of $\mT$, viewed as a single graph. The non-root $k$-cliques of this decoration correspond to the black children of $o$ in a canonical way, by using the total order on the black children and  ordering the non-root $k$-cliques of the decoration according to some rule or fixed choice. So we may proceed recursively by gluing for each black child of $v$ the corresponding $k$-clique of $\alpha(o)$ to the  root clique of the (connected graph corresponding to) the decoration $\alpha(v)$. We then proceed in the same fashion with the grandchildren of $o$, and so on. This is illustrated in Figure \ref{fig:blow_up}. Since this \emph{blow-up procedure} is a reformulation of Algorithm~\ref{alg_bfs}, the resulting graph $\mG^{(k)}_k$ is hence (up to relabelling of vertices) distributed like the Boltzmann random variable $\Gamma G_{k}^{(k)}(\rho_k)$. (Note that analogously we may transform any decorated tree into a graph via this blow-up operation.)

\begin{figure}[h]
    \centering
    \captionsetup[subfigure]{labelformat=empty}
    \begin{subfigure}[t]{0.5\textwidth}
        \centering
        \vtop{\vskip0pt \hbox{\includegraphics[width=\textwidth]{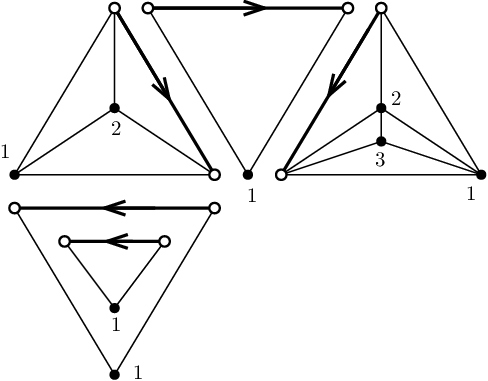}}}
    \end{subfigure}
    \hfill
    \begin{subfigure}[t]{0.42\textwidth}
        \centering
        \vtop{\vskip0pt \hbox{\includegraphics[width=\textwidth]{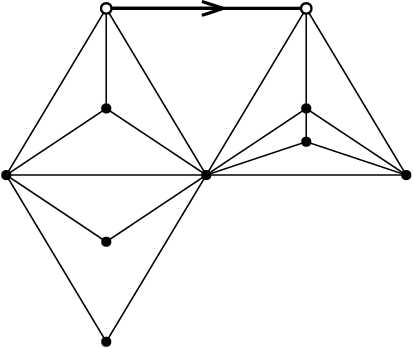}}}
    \end{subfigure}
    \caption{Blow-up of the tree in Figure \ref{fig:tree} with the decoration in Figure \ref{fig:tree_decorations}.}
    \label{fig:blow_up}
\end{figure}

Consequently, the random graph $\mG_{k,n}^{(k)}$ obtained by conditioning $\mG^{(k)}_k$ on having $n$ non-root vertices is uniformly distributed among all graphs from  $\mathcal{G}_{k,n}^{(k)}$. We may view $\mG_{k,n}^{(k)}$ as the result of forming a tree $\mT_n$ obtained by conditioning $\mT$ on having $n$ white vertices, choosing for each vertex $v$ of $\mT_n$ a decoration $\alpha_n(v)$ uniformly at random with such that the numbers of black and white children of $v$ match the numbers of non-root $k$-cliques and vertices of the decoration, and finally performing the same blow-up procedure as before.

\subsection{Structural analysis}

We collect some asymptotic properties that we are going to use frequently. 

\begin{proposition}
	\label{pro:offspring}
	\begin{enumerate} 
		\item The random non-negative integer $(\xi, \zeta)$ has finite exponential moments. That is, there exists $\epsilon>0$ such that
		\[
		\ex[ (1+ \epsilon)^\xi (1+\epsilon)^\zeta ]< \infty.
		\]
		\item We have 
		\[
		\ex[\xi] =  1.
		\]
		\item Letting $\#_1 (\cdot)$ and $\#_2 (\cdot)$ denote the number of black and white vertices of an input tree, we have
		\[
		\PP(\#_2 \mT = n) \sim	\frac{\boldsymbol \alpha_k k! \rho_k^{-k} }{G_k^{(k)}(\rho_k)} n^{-3/2}.
		\]
	\end{enumerate}
\end{proposition}
\begin{proof}
	The fact that $(\xi, \zeta)$ has finite exponential moments follows readily from Equation~\eqref{eq:subcrit}.  In order to see that $\ex[\xi]=1$, note that Equation~\eqref{eq:recursive} implies that
	\[
		\phi(x, G_k^{(k)}(x)) = 0
	\]
	for
	\[
		\phi(x,y) = y -  \exp\left( G_{k+1}^{(k)}\big(x, 1, \ldots, 1, y, 1, \ldots, 1 \big) \right).
	\]
	By the implicit function theorem, if $\frac{\partial \phi}{\partial y} (\rho_k, G_k^{(k)}(\rho_k)) \ne 0$, then $G_k^{(k)}(x)$ would have an analytic continuation in a neighbourhood of $\rho_k$, which contradicts Pringsheim's theorem~\cite[Thm. IV.6]{MR2483235}. Using~\eqref{eq:recursive}, it hence follows that
	\begin{align*}
	0 &= \frac{\partial \phi}{\partial y} (\rho_k, G_k^{(k)}(\rho_k)) \\
	&= 1 -  G_k^{(k)}(\rho_k)\frac{\partial G_{k+1}^{(k)}}{\partial x_k}\left(\rho_k, 1, \ldots, 1, G_k^{(k)}(\rho_k), 1, \ldots, 1\right) \\
	&= 1 - \ex[\xi].
	\end{align*}
	The probability for the event $\#_2 \mT= n$ is equal to the probability that the Boltzmann sampler $\Gamma G_k^{(k)}(\rho_k)$ produces a graph with $n$ non-root vertices. That is,
	\[
		\PP(\#_2 \mT = n) = \frac{ |\mathcal{G}_{k,n}^{(k)}| \rho_k^n }{n! G_k^{(k)}(\rho_k)}.
	\]
	Using Equations~\eqref{eq:asymptoticclique} and~\eqref{eq:asymnum}, it follows that
\begin{equation} 
	\label{eq:condprob}
\PP(\#_2 \mT = n) \sim	\frac{\boldsymbol \alpha_k k! \rho_k^{-k} }{G_k^{(k)}(\rho_k)} n^{-3/2}
\end{equation}
as $n\to\infty$.
\end{proof}

We collect a few additional properties concerning the degree profile of the black subtree of $\mT_n$.

\begin{proposition}
	\label{pro:degrees}
	\begin{enumerate}
	\item We have $\#_1 \mT_n / n \overset{p} \longrightarrow 1/\mathbb{E}[\zeta]$.
	
	\item For each integer $d \ge 0$ the number $B_d(\mT_n)$ of black vertices with  $d \ge 0$ black children satisfies $B_d(\mT_n) / n \overset{p} \longrightarrow \mathbb{P}(\xi=d)/\mathbb{E}[\zeta]$.
	
	\item There is a constant $C>0$ such that 
	\[
	\lim_{n \to \infty} \mathbb{P}\left(\sum_{d \ge C \log n} B_d(\mT_n)  =0\right) = 1.
	\]
	
	\item We have $\sum_{d \ge 1} d^2 \mathbb{E}[\zeta] B_d(\mT_n) / n  \overset{p} \longrightarrow 1 +  \mathbb{V}[\xi]$.
	
	\end{enumerate}
\end{proposition}
\begin{proof}

	We may construct the black subtree of $\mT$ in a depth-first-search order so that the number of black and white children of the $i$th vertex in that order are given by an independent copy $(\xi_i, \zeta_i)$ of $(\xi,\zeta)$. By standard properties of the depth-first-search the tree is fully explored after $\ell$ steps if $\sum_{i=1}^{\ell} (\xi_i -1) = -1$ and $\sum_{i=1}^k (\xi_i -1) \ge 0$ for all $1 \le k < \ell$. See for example~\cite[Lem. 15.2]{JANSON2012} The total number of white vertices is then given by $\sum_{i=1}^\ell \zeta_i$, and the number of black vertices with $d$ black children by $\sum_{i=1}^\ell \mathbf{1}_{\xi_i= d}$. Using the well-known cycle lemma~\cite[Lem. 15.3]{JANSON2012}, it follows that
	\begin{align*}
		\mathbb{P}( \#_1 &\mT= \ell, \#_2 \mT= n)  \\&= \mathbb{P}\left(\sum_{i=1}^{\ell} (\xi_i -1, \zeta_i) = (-1,n),  \sum_{i=1}^k (\xi_i -1) \ge 0\text{  for all } 1 \le k < \ell \right) \\
		&=\frac{1}{\ell} \mathbb{P}\left(\sum_{i=1}^{\ell} (\xi_i -1, \zeta_i) = (-1,n) \right).
	\end{align*}
	We have shown above that $\zeta$ has finite exponential moments. Hence, by a standard medium deviation inequality recalled in Lemma~\ref{le:deviation} it follows that 
	\[
	\mathbb{P}\left(\sum_{i=1}^\ell \zeta_i = n\right) \le O(1) \exp(- \Theta(|n/\mathbb{E}[\zeta] - \ell | / \sqrt{\ell}))
	\]
	uniformly for all integers $\ell$. For any $0< \delta < 1/4$ we may set \[
	I = \{ \ell \mid |\ell - n / \mathbb{E}[\zeta] | < n^{1/2 + \delta } \}.
	\] Using Equation~\eqref{eq:condprob} we obtain
	\begin{align}
		\label{eq:stretch}
		\mathbb{P}( \#_1 \mT_n \notin I)  &= \mathbb{P}( \#_1 \mT \notin I \mid  \#_2 \mT= n) \nonumber\\
		&\le O(n^{3/2}) \mathbb{P}( \#_1 \mT \notin I ,  \#_2 \mT= n) \nonumber\\
		&\le O(n^{3/2})\sum_{\ell \notin I} \frac{1}{\ell} \exp(- \Theta(|n/\mathbb{E}[\zeta] - \ell | / \sqrt{\ell})) \nonumber\\
		&= \exp(- \Theta(n^{\delta})). 
	\end{align}
	In particular, 
	\[
	\#_1 \mT_n  / n \overset{p}\longrightarrow 1 / \mathbb{E}[\zeta].
	\]
	For notational convenience, we set\[
	I_d = [\mathbb{P}(\xi=d) n / \mathbb{E}[\zeta] - n^{1/2 + 2 \delta}, \mathbb{P}(\xi=d) n / \mathbb{E}[\zeta] + n^{1/2 + 2 \delta}].
	\] Using Equation~\eqref{eq:condprob} it follows that
	\begin{align*}
		\mathbb{P}( B_d(\mT_n) \notin I_d ) &= o(1) + \sum_{\ell \in I} \mathbb{P}(B_d(\mT_n) \notin I_d, \#_1 \mT_n = \ell) \\
		&=  o(1) + \sum_{\ell \in I} \frac{\mathbb{P}(B_d(\mT) \notin I_d, \#_1 \mT = \ell, \#_2 \mT= n)}{\mathbb{P}(\#_2 \mT= n)} \\
		&\le o(1) +  O(n^{3/2})   \sum_{\ell \in I}\mathbb{P}\left(\sum_{i=1}^\ell \mathbf{1}_{\xi_i= d} \notin I_d, \sum_{i=1}^\ell \zeta_i= n \right).
	\end{align*}
	 Applying the Chernoff bounds it follows that \[
	\mathbb{P}\left( \sum_{i=1}^\ell \mathbf{1}_{\xi_i= d} \notin I_d \right) \le \exp\left( - \Theta(n^{\delta}) / \mathbb{P}(\xi=d) \right)
	\]
	uniformly for all $d \ge 0$ and all  integers $\ell$ with  $|\ell - n/\mathbb{E}[\zeta]| < n^{\delta/2}$. It follows that
	\[
	\mathbb{P}( B_d(\mT_n) \notin I_d ) \le \exp(-\Theta(n^{\delta/2}))
	\]
	uniformly for all $d$. In particular,
	\[
	 B_d(\mT_n) / n \overset{p} \longrightarrow \mathbb{P}(\xi=d)/\mathbb{E}[\zeta].
	\]
	
	Next, arguing analogously as before we obtain for any $d_0 \ge 0$
	\begin{align*}
	\mathbb{P}\left( B_d(\mT_n) >0\text{ for some } d \ge d_0 \right) &\le o(1) + O(n^{3/2})   \sum_{\ell \in I}\mathbb{P}\left(\sum_{i=1}^\ell \mathbf{1}_{\xi_i\ge d_0} >0 \right) \\
	&\le o(1) + O(n^{7/2})  \mathbb{P}(\xi \ge d_0). \\
	\end{align*}
	Since $\xi$ has finite exponential moments, we may set $d_0 = C \log n$ for a large enough constant $C>0$ such that this probability tends to zero. This shows:
		\[
	\lim_{n \to \infty} \mathbb{P}\left(\sum_{d \ge C \log n} B_d(\mT_n)  =0\right) = 1.
	\]
	
	Finally, with the intermediate results at hand, we have with probability tending to $1$ as $n \to \infty$ that
	\begin{align*}
		\sum_{d \ge 1} d^2 \mathbb{E}[\zeta] B_d(\mT_n) / n   &=  \sum_{d = 1}^{C \log n} d^2 ( \mathbb{P}(\xi=d) + O(n^{-1/4}) ) \\
		&= \mathbb{V}[\xi] +1 + o(1).
	\end{align*}
	This concludes the proof. 
\end{proof}

We may order the vertices of $\mT_n$ linearly according to a depth-first-search exploration. This induces an ordering of the $\#_1 \mT_n$ black vertices and an ordering on the $\#_2 \mT_n =n$ white vertices. Let $U$ denote a random element of the half-open interval $]0,1]$ sampled according to the uniform distribution. Hence $L_{n} := \lceil U \#_1 \mT_n \rceil$ is distributed like the position of a uniformly selected black vertex of $\mT_n$ in the induced ordering on the black vertices. Likewise, we may select the white vertex at position $\lceil U n \rceil$ among the $n$ white vertices (so that it is uniformly distributed) and let $1 \le L_n' \le \#_1 \mT_n$ denote the position of its unique black parent in the ordering of the $\#_1 \mT_n$ black vertices. 

\begin{proposition}
	\label{pro:coupling}
Given $0 < \delta < 1/4$, there exists a constant $C>0$ such that for all $n$ the tree $\mT_n$ has with probability at least $1 - \exp(-\Theta(n^\delta))$ the property that
	\[
		|L_n - L_n'| \le C n^{3/4},
	\]
regardless of the value of $U$.
\end{proposition}
\begin{proof}
	We use the same notation as in the proof of Proposition~\ref{pro:degrees}. 
	By Equation~\eqref{eq:stretch} we know that for any $0< \delta < 1/4$ 
	\[
		\mathbb{P}(\#_1\mT_n \notin I) \le \exp(- \Theta(n^\delta))
	\]
	with $
	I = \{ \ell \mid |\ell - n / \mathbb{E}[\zeta] | < n^{1/2 + \delta }\}$.  Consider the event $\mathcal{E}$ that there exists an index $1 \le j \le \#_1 \mT_n$ such that the sum $S$ of the number of white children of the first $j$ black vertices in depth-first-search order does not lie in $J := \{ r \mid |j \mathbb{E}[\zeta] -r | \le   n^{1/2 + \delta}|$. By analogous arguments as for Equation~\eqref{eq:stretch}, we obtain
	\begin{align*}
		\mathbb{P}(\mathcal{E}, \#_1 \mT_n \in I) & \le O(n^{3/2}) \sum_{\ell \in I} \sum_{j=1}^\ell \mathbb{P}\left( \left| \sum_{i=1}^j \zeta_i - j \mathbb{E}[\zeta] \right| > n^{1/2 + \delta}\right) \\&\le O(n^{3/2}) \sum_{\ell \in I} \sum_{j=1}^\ell \sum_{r\notin J} \exp( - \Theta( |j \mathbb{E}[\zeta] - r| / \sqrt{j} )) \\
		&\le \exp(-\Theta(n^{\delta})).
	\end{align*}

	Suppose now that $\#_1 \mT_n \in I$ and that the event $\mathcal{E}$ does not take place. 
	Then, by construction, we know that
	\[
		\frac{L_n}{\#_1{\mT_n}} = \frac{S}{n} + O(1/n)
	\]
	and $|\#_1{\mT_n} - n / \mathbb{E}[\zeta]| \le  n^{1/2 + \delta}$ and $|S - L_n' \mathbb{E}[\zeta]| \le  n^{1/2 + \delta}$, which implies
	\[
		|L_n - L_n'| = O(n^{1/2 + \delta}	).
	\]

\end{proof}

\subsection{De-rooting the $k$-clique rooted random graphs}

The random graph $\mG_{k,n}^{(k)}$ is uniformly distributed among all $k$-connected chordal graphs with tree-width at most $t$ that are rooted at a $k$-clique and having $n$ vertices not in the
root clique.

\begin{lemma}
	\label{le:deroot}
	If we consider $\mG_{k,n}^{(k)}$ and $\mG_{k,n+k}$ as unrooted unlabelled graphs by forgetting about the labels and the root-clique, then
	\[
		d_{\mathrm{TV}}(\mG_{k,n}^{(k)}, \mG_{k,n+k}) \to 0
	\]
	as $n \to \infty$.
\end{lemma}
\begin{proof}
	Let $\mathrm{L}(\mG_{k,n}^{(k)})$ denote the result of relabelling all vertices of $\mG_{k,n}^{(k)}$ (including those contained in the root clique) uniformly at random. Furthermore, let $\mathrm{F}(\cdot)$ denote a forgetful functor that takes as input a clique-rooted graph and outputs the unrooted graph obtained by forgetting about the root clique. 
	
	By~\eqref{eq:cliqueclt} the subset $\mathcal{E}_{k, n}$ of all all graphs $G \in  \mathcal{G}_{k,n}$ with
	\[
		| n_k(G) - \boldsymbol \alpha_k n | \le n^{3/4}
	\]
	has the property that $\mathsf{G}_{k,n} \in \mathcal{E}_{k, n}$ with high probability as $n \to \infty$. Using Equation~\eqref{eq:asymptoticclique}, it follows that 
	\begin{align*}
		\frac{\PP(\mathrm{F}(\mathrm{L}(\mG_{k,n}^{(k)})) = G)}{\PP(\mathsf{G}_{k,n+k} = G)} &= \frac{n_k(G) / ( \binom{n+k}{k} |\mathcal{G}_{k,n}^{(k)}|) }{1 / |\mathcal{G}_{k,n+k}| } \to 1
	\end{align*}
	uniformly for all $G \in \mathcal{E}_{k, n+k}$ as $n \to \infty$. This completes the proof.
\end{proof}

\section{Proof of the local limit}

\subsection{Local convergence of random trees}

Consider the set $\mathfrak{X}_f$ of all finite $2$-type plane trees with a root vertex and second marked vertex (which is not necessarily distinct from the root vertex), subject to the following restriction:  We require all trees in $\mathfrak{X}_f$ to have the property that vertices of the second type have no offspring. For ease of reference, we will again refer to vertices of the first type as black vertices, and vertices of the second type as white vertices.

Let $T$ denote a $2$-type plane tree. For any vertex $v$ of $T$ we can form the \emph{fringe subtree} $f(T,v)$ consisting of all descendants of $v$, including $v$ itself.   For any vertex $v$ of $T$ and any integer $h \ge 0$ we can define the \emph{extended fringe subtree} $f^{[h]}(T,v) \in \mathfrak{X}_f$ obtained by taking the fringe subtree $f(T, v_h)$ of $T$ at the $h$-th ancestor $v_h$ of $v$, and marking  it at the vertex of $f(T, v_h)$ that corresponds to $v$. Of course, this is only well-defined if $v$ has height at least $h$ in $T$.  If $v$ has smaller height, we set $f^{[h]}(T, v)$ to some place-holder value. Given $\tau \in \mathfrak{X}_f$ we define the number $N_\tau(T)$ as the number of vertices $v$ of the tree $T$ such that $f^{[h_\tau]}(T,v) = \tau$, with $h_\tau$ denoting the height of the marked vertex of $\tau$. We also say $N_\tau(T)$ counts the \emph{occurrences} of $\tau$ in $T$.

Let $\mathfrak{X}$ denote the union of $\mathfrak{X}_f$ with the set of all locally finite $2$-type plane trees with a marked vertex having a countably infinite number of ancestors, subject to the following restriction: We  require all trees in $\mathfrak{X}$ to have the property that the marked vertex is white,  that all white vertices to have no offspring, and that all extended fringe subtrees of the marked vertex are finite. In particular, any tree in $\mathfrak{X}$ may be decomposed into a (possibly infinite) \emph{spine} given by the ancestors of the marked vertex, to which finite trees are attached.

The space $\mathfrak{X}$ can be equipped with a metric
\[
	d_{\mathfrak{X}}(\tau, \tau') = 2^{- \sup\{h \ge 0 \mid f^{[h]}(\tau) = f^{[h]}(\tau')\}  }, \qquad \tau, \tau' \in \mathfrak{X}.
\]
It is easy to see that this metric is complete and separable~\cite[Prop. 1]{MR4414401}. Weak convergence of probability measures defines a topology on $\mathfrak{M}(\mathfrak{X})$, the space of probability measures on $\mathfrak{X}$, making it a Polish space too.

We let $v_n$ denote a uniformly selected white vertex of our random tree~$\mT_n$. This makes $(\mT_n, v_n)$ a random element of $\mathfrak{X}$. We can also interpret the conditional law $\mathfrak{L}( (\mT_n, v_n) \mid \mT_n )$ as a random element of $\mathfrak{M}(\mathfrak{X})$.

 That is, for any tree $T \in \mathfrak{X}$ with $n$ white vertices we can look at the uniform probability measure that assigns probability $1/n$ to each of the $n$ marked versions of $T$. This measure is a (deterministic) element of $\mathfrak{M}(\mathfrak{X})$. 
 The conditional law $\mathfrak{L}( (\mT_n, v_n) \mid \mT_n )$ may be constructed in this way by taking $T = \mT_n$ to be random, and hence it is a  random element of $\mathfrak{M}(\mathfrak{X})$.

Proposition~\ref{pro:offspring} allows us to apply a general result~\cite[Thm. 1]{MR4414401}, yielding that there exists a random limiting tree $\mT^\circ$ with
\begin{align}
	(\mT_n, v_n)\stackrel{d}{\longrightarrow} \mT^\circ
\end{align}
and
\begin{align}
	\label{eq:treequenched}
	\mathfrak{L}( (\mT_n, v_n) \mid \mT_n)  \stackrel{p}{\longrightarrow} \mathfrak{L}(\mT^\circ).
\end{align}
The first equation is referred to as annealed local convergence, and the second (stronger) equation as quenched local convergence.

The limiting tree $\mT^\circ$ may be described as follows. We first define a size-biased versions $(\xi^\bullet, \zeta^\bullet)$ and $(\xi^\circ, \zeta^\circ)$ of the random variable $(\xi, \zeta)$ by
\begin{align}
	\PP\left( (\xi^\circ, \zeta^\circ)  = (a,b)\right) = \PP\left( (\xi,\zeta) = (a,b)\right) b  /\ex[\zeta]
\end{align}
and
\begin{align}
	\PP\left( (\xi^\bullet, \zeta^\bullet)  = (a,b)\right) = \PP\left( (\xi,\zeta) = (a,b)\right) a.
\end{align}
Now, let $u_1, u_2, \ldots$ denote a sequence of black vertices, which we are going to call the \emph{spine} of $\mT^\circ$. The tip $u_1$ of the spine receives offspring according to $(\xi^\circ, \zeta^\circ)$ and a uniformly selected white child $u_0$ is marked. For each $i \ge 2$ the vertex $u_i$ receives offspring according to an independent copy of $(\xi^\bullet, \zeta^\bullet)$ and we identify $u_{i-1}$ with a uniformly selected black child. The construction of $\mT^\circ$ is finalized by identifying each black non-spine vertex with the root of an independent copy of the unconditioned tree $\mT$.

\subsection{Local convergence of random decorated trees}

Recall from Section~\ref{sec:blowup} that the random decorated tree $(\mT_n, \alpha_n)$ is constructed from $\mT_n$ by drawing for each black vertex $v$ of $\mT_n$ an independent decoration $\alpha_n(v)$  uniformly at random from the class $\mathrm{SET}(\mathcal{G}_{k+1}^{(k)})$ such that the numbers of non-root $k$-cliques and vertices of $\alpha(v)$ agree with the numbers of black and white children of $v$. We refer to this construction as the \emph{canonical way} of decorating the tree $\mT_n$. Of course, this can be performed on any $2$-type plane tree, allowing us to form the decorated tree $(\mT^\circ, \alpha^\circ)$ by assigning random decorations $\alpha^\circ(v)$, $v \in \mT^\circ$ in the canonical way. 

We define the space $\mathfrak{A}$ as the collection of all tuples $(\tau, \alpha_\tau)$ with $\tau \in \mathfrak{X}$ and $\alpha_\tau$ a (deterministic) decoration of $\tau$. That is, $\alpha_\tau= (\alpha_\tau(v))_{v \in \tau}$ is a family of objects from $\mathrm{SET}(\mathcal{G}_{k+1}^{(k)})$ so that for each $v \in \tau$ the numbers of non-root $k$-cliques and vertices of $\alpha_\tau(v)$ agree with the numbers of black and white children of $v$ in $\tau$. (Here we use the index $\tau$ in the notation $\alpha_\tau$ only in order to avoid any confusion with the canonical decoration $\alpha$ of the unconditioned Bienaym\'e--Galton--Watson tree~$\mT$. Of course, multiple decorated versions may correspond to the undecorated tree $\tau$.) We define the subset $\mathfrak{A}_f \subset \mathfrak{A}$ of all such trees $(\tau, \alpha_\tau)$ that are finite.

The notions of extended fringe subtrees and the metric on $\mathfrak{X}$ may be defined analogously for decorated trees, making $\mathfrak{A}$ a Polish space. Analogously to the number $N_\tau(T)$ of occurrences of a finite  marked tree $\tau$ in some plane tree $T$, we may define the number $N_{(\tau, \alpha_\tau)}(T, \alpha_T)$ of occurrences of a finite decorated tree marked tree $(\tau, \alpha_\tau)$ in the tree $T$ endowed with some decoration $\alpha_T$. 

Equation~\eqref{eq:treequenched} may be used to deduce quenched local convergence of random decorated trees:
\begin{proposition}
	\label{pro:quencheddeco}
	As $n \to \infty$, 
	\begin{align*}
		\mathfrak{L}( ((\mT_n, v_n), \alpha_n) \mid (\mT_n,\alpha_n) )  \stackrel{p}{\longrightarrow} \mathfrak{L}(\mT^\circ, \alpha^\circ).
	\end{align*}
\end{proposition}

\begin{proof}
It is easy but crucial to observe that occurrences of $\tau$ in $T$ are pairwise disjoint. Therefore, if the decoration $\alpha_T$ of $T$ is chosen at random in the mentioned canonical way, then each occurrence of $\tau$ in $T$ has a fixed probability $p_{(\tau, \alpha_\tau)}$ of also being an occurrence of $(\tau, \alpha_\tau)$, independently from the other occurrences. This probability $p_{(\tau, \alpha_\tau)}$ is precisely equal to the probability that a random canonical decoration of $\tau$ is equal to $\alpha_\tau$.

The quenched local convergence from Equation~\eqref{eq:treequenched} is equivalent to stating that for each $\tau \in \mathfrak{X}_f$
\[
	\frac{N_\tau(\mT_n)}{n} \stackrel{p}{\longrightarrow} p_\tau
\]
for $p_\tau := \PP(f^{[h_\tau]}(\mT^\circ) = \tau)$, with $h_\tau$ denoting the height of the marked vertex of $\tau$. Since occurrences of $\tau$ do not overlap, it follows by the Chernoff bounds that for each decorated version $(\tau, \alpha_\tau)$ of $\tau$ we have
\[
	\frac{N_{(\tau, \alpha_\tau)}(\mT_n, \alpha_n)}{n} \stackrel{p}{\longrightarrow} p_\tau p_{(\tau, \alpha_\tau)}.
\]
We have 
\begin{align*}
p_\tau p_{(\tau, \alpha_\tau)} &= \PP(f^{[h_\tau]}(\mT^\circ) = \tau) \PP(f^{[h_\tau]}(\mT^\circ, \alpha^\circ) 
= (\tau, \alpha_\tau) \mid  f^{[h_\tau]}(\mT^\circ) = \tau) \\
 &= \PP(f^{[h_\tau]}(\mT^\circ, \alpha^\circ) = (\tau, \alpha_\tau)).
\end{align*}
It follows that
	\begin{align*}
	\mathfrak{L}( ((\mT_n, v_n), \alpha_n) \mid (\mT_n,\alpha_n) )  \stackrel{p}{\longrightarrow} \mathfrak{L}(\mT^\circ, \alpha^\circ)
\end{align*}
by the de-rooting.

\end{proof}

\subsection{Local convergence of random graphs}

The blow-up procedure described in Section~\ref{sec:blowup} constructs a graph from a decorated tree. Consequently, we may assign to each marked decorated tree $(\tau, \alpha_\tau) \in \mathfrak{A}$ a vertex-rooted graph $G$ obtained by performing these blow-up operations on $(\tau, \alpha_\tau)$. In each step of this procedure we join a graph constructed so far to some previously unused decoration by gluing them together at specific $k$-cliques. The graph $G$ is locally finite if and only if by completing the procedure we never glue a vertex an infinite number of times to other vertices. We let $\mathfrak{A}_0 \subset \mathfrak{A}$ denote the subset of decorated marked trees where the corresponding graph is locally finite. We define a map
\[
	\Phi: \mathfrak{A} \to \mathfrak{G}, 
\]
that maps a marked decorated tree $(\tau, \alpha_\tau)$ to a neighbourhood $U_{j_0}(G)$ of the  corresponding vertex-rooted graph $G$ with
\[
	j_0 = \sup\{j \ge 0 \mid U_j(G) \text{ is finite} \} \in \{0, 1, 2, \ldots\} \cup \{\infty\}.
\]
This way, $\Phi(\tau, \alpha_\tau) = G$ if $(\tau, \alpha_\tau) \in \mathfrak{A}_0$.

\begin{proposition}
	\label{pro:annealedprep}
	The following statements hold.
	\begin{enumerate}
		\item The map $\Phi$ is continuous in a point $(\tau, \alpha_\tau)$ if and only if $(\tau, \alpha_\tau) \in \mathfrak{A}_0$. 
		\item  $\mathfrak{A}_0$ is Borel measurable and the map $\Phi$ is measurable.
		\item  $(\mT^\circ, \alpha^\circ) \in \mathfrak{A}_0$ holds almost surely.
	\end{enumerate}
\end{proposition}
\begin{proof}
	We start with the first claim. Only the vertices that are part of $k$-cliques that correspond to the black vertices of the spine of an infinite marked decorated tree $(\tau, \alpha_\tau)$ from $\mathfrak{A}$ may undergo an infinite number of gluing when forming the corresponding graph $G$. Such a vertex $v$ has infinite degree in $G$ if and only if all but a finite number of these cliques are glued together in an overlapping manner so that the overlap always contains $v$. 
		
	It follows that if $(\tau, \alpha_\tau) \in  \mathfrak{A}_0$ then for each $j \ge 0$ there exist a number $h(j) \ge 0$ such that the graph corresponding to  $f^{[h(j)]}(\tau, \alpha_\tau)$ already contains the $j$-neighbourhood $U_j(\Phi(\tau, \alpha_\tau))$. Thus, $\Phi$ is continuous in $\mathfrak{A}_0$. 
	
	On the other hand, suppose that $(\tau, \alpha_\tau) \in \mathfrak{A} \setminus \mathfrak{A}_0$, so that  $j_0$ as defined as above is finite.  $\Phi$ cannot be continuous in $(\tau, \alpha_\tau)$ since $f^{[h]}(\tau, \alpha_\tau) \to (\tau, \alpha_\tau)$ as $h \to \infty$, but clearly $\Phi(f^{[h]}(\tau, \alpha_\tau))$ has radius at least $j_0 + 1$ for all sufficiently large $h$.
This concludes the proof of the first claim.
	
	We proceed to prove the second claim. The points of continuity of a map defined on a metric space always forms a Borel measurable subset, see~\cite[Appendix M, Sec. M10]{MR1700749}. Consequently, the subset $\mathfrak{A}_0 \subset \mathfrak{A}$ is measurable.  Using the fact that any open set in $\mathfrak{G}$ is a countable union of sets of the form $\{G \in \mathfrak{G} \mid U_m(G) = H\}$, $m \ge 0$, $H$ a finite rooted graph, measurability of $\Phi$ now follows by standard arguments.
	
	We now prove the third claim. Let $u_1, u_2, \ldots$ denote the black spine vertices of $\mT^\circ$, so that $u_{i+1}$ is the father of $u_i$ for all $i \ge 1$. Denote by $\mathcal{E}_i$ the event that the spine black vertices $u_i$ and $u_{i+1}$ have decorations $\alpha^\circ(u_i)$ and $\alpha^\circ(u_{i+1})$ satisfying that the two root-cliques in them are disjoint after the gluing. We know by the discussion in the first paragraph that in order for $(\mT^\circ, \alpha^\circ)$ to be in $ \mathfrak{A}_0$ it is sufficient that $\mathcal{E}_i$ happens for infinitely many $i$. Since $\mathcal{E}_i$ has a fixed positive probability that does not depend on $i$ and $(\mathcal{E}_i)_{i \ge 1}$ is independent, it follows by the Borel-Cantelli lemma that almost surely $\mathcal{E}_i$ takes place for infinitely many $i$. Thus, $(\mT^\circ, \alpha^\circ) \in \mathfrak{A}_0$ holds almost surely. This completes the proof.
\end{proof}

\begin{corollary}
	\label{co:quenchedprep}
	The push-forward map
	\[
		\Phi^*: \mathfrak{M}(\mathfrak{A}) \to \mathfrak{M}(\mathfrak{G}), \qquad P \mapsto P \Phi^{-1}
	\]
	that maps a probability measure $P$ to the push-forward measure  $P \Phi^{-1}: B \mapsto P(\Phi^{-1}(B))$ has the following properties.
	\begin{enumerate}
		\item $\Phi^*$ is continuous at the point $\mathfrak{L}(\mT^\circ, \alpha^\circ)$.
		\item $\Phi^*$ is measurable.
	\end{enumerate}
\end{corollary}
\begin{proof}
	Proposition~\ref{pro:annealedprep} allows us to apply the  continuous mapping theorem~\cite[Thm. 2.7]{MR1700749}, yielding directly that $\Phi^*$ is continuous at $\mathfrak{L}(\mT^\circ, \alpha^\circ)$. Measurability of the map $\Phi^*$ follows by standard arguments.
\end{proof}

We are now ready to prove Theorem~\ref{te:local}.

\begin{proof}[Proof of Theorem~\ref{te:local}]
Recall that by Proposition~\ref{pro:quencheddeco} we have
\begin{align*}
	\mathfrak{L}( ((\mT_n, v_n), \alpha_n) \mid (\mT_n,\alpha_n) )  \stackrel{p}{\longrightarrow} \mathfrak{L}(\mT^\circ, \alpha^\circ)
\end{align*}
as $n \to \infty$. Corollary~\ref{co:quenchedprep} allows us to apply the  continuous mapping theorem~\cite[Thm. 2.7]{MR1700749}, yielding
\begin{align*}
		\Phi^*\mathfrak{L}( ((\mT_n, v_n), \alpha_n) \mid (\mT_n,\alpha_n) )  \stackrel{p}{\longrightarrow} \Phi^*\mathfrak{L}(\mT^\circ, \alpha^\circ).
\end{align*}
Clearly $\Phi^*\mathfrak{L}(\mT^\circ, \alpha^\circ)$ is equal to the law $\mathfrak{L}(\hat{\mG})$ of the infinite random rooted graph $\hat{\mG} := \Phi(\mT^\circ, \alpha^\circ)$. (Of course, $\hat{\mG}$ depends implicitly on $k$.)
Recall that $\mG_{k,n}^{(k)}$ denotes a uniformly selected labelled $k$-clique rooted $k$-connected chordal graph with tree-width at most $t$. 
 The push-forward $\Phi^*\mathfrak{L}( ((\mT_n, v_n), \alpha_n) \mid (\mT_n,\alpha_n) )$ is distributed like the conditional law $\mathfrak{L}( (\mG_{k,n}^{(k)}, x_n^\circ) \mid \mG_{k,n}^{(k)})$, with $x_n^\circ$ denoting a uniformly selected non-root vertex of $\mG_{k,n}^{(k)})$. If we select a vertex $x_{n}$ uniformly at random among the $n+k$ vertices of $\mG_{k,n}^{(k)})$, then this vertex is with high probability not one of the $k$ vertices contained in the root-clique. Consequently, it follows that
 \begin{align*}
 	\mathfrak{L}( (\mG_{k,n}^{(k)}, x_{n}) \mid \mG_{k,n}^{(k)}) \stackrel{p}{\longrightarrow}  \mathfrak{L}(\hat{\mG}).
 \end{align*}
By Lemma~\ref{le:deroot} it follows that
 \begin{align*}
	\mathfrak{L}( (\mG_{k,n}, y_{n}) \mid \mG_{k,n}^{(k)}) \stackrel{p}{\longrightarrow}  \mathfrak{L}(\hat{\mG}),
\end{align*}
with $y_n$ denoting a uniformly selected vertex of $\mG_{k,n}$. This completes the proof.
\end{proof}

\section{Proof of the scaling limit}

\subsection{Scaling limit of a coupled random tree}

We are going to verify the following limit of the underlying tree $\mT_n$, which is a special case of a conditioned sesqui-type tree. That is, a tree with two types, only one of which produces offspring.

\begin{proposition}
	\label{pro:scaling limit}
	 There is a constant 
	 \[
	 	\kappa_{\mathrm{tree}} = \frac{\sqrt{\mathbb{V}[\xi] \mathbb{E}[\zeta]}}{2}
	 \]
	that depends on $k$ and $t$ such that
	\[
	\left( \mT_{n}, \kappa_{\mathrm{tree}}n^{-1/2}d_{\mT_{n}}, \mu_{\mT_{n}} \right) \stackrel{d}{\longrightarrow}
	\left( \mathscr{T}_e, d_{\mathscr{T}_e}, \mu_{\mathscr{T}_e} \right)
	\]
	in the Gromov--Hausdorff--Prokhorov sense as $n \to \infty$. Here $\mu_{\mT_{n}}$ denotes the uniform measure on the white vertices of $\mT_n$. Furthermore, there are constants $C,c>0$ such that for all $n$ and all  $x \ge 0$ the height $\mathrm{H}(\mT_n)$ satisfies
	\[
		\PP(\mathrm{H}(\mT_n) \ge x) \le C \exp(-c x^2 / n).	
	\]
\end{proposition}



\begin{proof}
	Let us call the outdegree profile of a rooted tree the family $(N_d)_{d \ge 0}$ with $N_d$ denoting the number of vertices with outdegree $d \ge 0$. Conversely, we can fix a profile and uniformly sample a rooted tree with that profile.
	
	Given an integer $1 \le \ell \le n$, the black subtree of the  conditioned tree $(\mT_n \mid \#_1\mT_n=\ell)$ is a mixture of uniform random $\ell$-vertex rooted trees with given outdegree profile. This is because if we condition $(\mT_n \mid \#_1\mT_n=\ell)$ on having a $2$-dimensional outdegree profile $(N_{(a,b)})_{a, b \ge 0}$, then the black subtree is uniformly distributed among all rooted trees with degree profile $( \sum_{b \ge 0} N_{(a,b)})_{ a \ge 0}$.\footnote{The authors are grateful to Louigi Addario-Berry for pointing out this fact.}
	
	This allows us to apply~\cite[Thm. 3]{addarioberry2022random} to the black subtree of $(\mT_n \mid \#_1\mT_n=\ell)$.  Any vertex of $\mT_n$ is at graph distance at most $1$ from a black vertex, hence~\cite[Thm. 3]{addarioberry2022random} in fact yields that there are constants $C,c>0$ that do not depend on $\ell$ or $n$  or $x \ge 0$ such that
	\[
		\mathbb{P}\left( (\mathrm{H}(\mT_n) \mid \#_1\mT_n=\ell) \ge x\right) \le C \exp(-cx^2 / \ell).
	\]
	This immediately yields
	\begin{align*}
		\PP(\mathrm{H}(\mT_n) \ge x) &= \sum_{\ell=1}^n \mathbb{P}(\#_1\mT_n=\ell)\mathbb{P}\left( (\mathrm{H}(\mT_n) \mid \#_1\mT_n=\ell) \ge x\right) \\
		&\le C \sum_{\ell=1}^n \mathbb{P}(\#_1\mT_n=\ell) \exp(-cx^2/\ell) \\
		&\le C \exp(-c x^2 / n).
	\end{align*}
	
	This concludes the proof of the tail-bounds for the height.  For the scaling limit, Proposition~\ref{pro:degrees} ensures that we may apply~\cite[Thm. 1]{MR3188597} to the black subtree $\mB_n$ of $\mT_n$, yielding
	\[
		\left(\mB_n,\frac{\sqrt{\mathbb{V}[\xi]}}{2} \sqrt{\frac{\mathbb{E}[\zeta]}{n}} d_{\mB_n}, \mu_{\mB_n}\right) \overset{d}\longrightarrow \left( \mathscr{T}_e, d_{\mathscr{T}_e}, \mu_{\mathscr{T}_e} \right).
	\]
	Here $\mu_{\mB_n}$ denotes the uniform probability measure on the set of vertices of $\mB_n$. Specifically, we obtain this limit  by using Skorokhod's representation theorem analogously as for monotype trees in~\cite[Sec. 6]{MR3188597}. That is, we again view $\mB_n$ as a mixture of trees with given degree sequences and  apply Skorokhod's representation theorem to Proposition~\ref{pro:degrees} to ensure the existence of a coupling so that pathwise the degree sequences satisfy the conditions of~\cite[Thm. 1]{MR3188597}. Hence we may apply~\cite[Thm. 1]{MR3188597} to obtain pathwise distributional convergence in the Gromov--Hausdorff--Prokhorov sense with the same rescaling factor for each path, yielding the scaling limit for $\mB_n$.
	
	To be precise, \cite[Thm. 1]{MR3188597} states Gromov--Hausdorff convergence, however it is clear that the proof given there extends to the Gromov--Hausdorff--Prokhorov metric. One way to see this is to combine the convergence of the contour function~\cite[Thm. 3]{MR3188597} with the fact~\cite[Prop. 2.10]{MR3201925} that the construction of trees from contour functions is continuous with respect to the Gromov--Hausdorff--Prokhorov metric.
	
	Any vertex of $\mT_n$ is at graph distance at most $1$ from $\mB_n$, hence it is clear that the Hausdorff distance between $\mB_n$ and $\mT_n$ tends to zero after rescaling distances by $n^{-1/2}$. It remains to bound the Prokhorov distance between $\mu_{\mB_n}$ (interpreted as a measure on the entire tree $\mT_n$ ) and the measure $\mu_{\mT_{n}}$, both after rescaling distances by $\kappa_{\mathrm{tree}} n^{-1/2}$.
	
	Proposition~\ref{pro:coupling} ensures that there exists a set $\mathcal{E}$ of trees so that 
	\begin{align}
		\label{eq:bcc}
	 | 1 - \mathbb{P}(\mT_n \in \mathcal{E}) | \le \exp(- \Theta(n^{2/3}))
	\end{align}
	and such that whenever $\mT_n \in \mathcal{E}$ we may   couple a uniformly selected black vertex $v_n$ of $\mT_n$ with the unique black parent $v_n'$ of a uniformly selected white vertex of $\mT_n$ such that the positions $L_n$ and $L_n'$  of $v_n$ and $v_n'$ in the depth-first-search order of $\mB_n$ satisfy 
	\begin{align}
		\label{eq:n34}
		|L_n - L_n'| \le n^{3/4}.
	\end{align}

	By applying \cite[Thm. 3]{MR3188597} to $\mB_n$ in the same way as we applied \cite[Thm. 1]{MR3188597} before, we know that the height process and the contour process of $\mB_n$ admit the same Brownian excursion of duration $1$ as distributional scaling limit after multiplying  height by $\kappa_{\mathrm{tree}} / \sqrt{n}$, and time by $\mathbb{E}[\zeta]/n$  and $\mathbb{E}[\zeta]/(2n)$, respectively, and jointly the \L{}ukasiewicz path converges in distribution towards the same Brownian excursion after rescaling height by $\sqrt{\frac{\mathbb{E}[\zeta]}{\mathbb{V}[\xi]}} \frac{1}{ \sqrt{n}}$ and time by $\mathbb{E}[\zeta]/n$. Skorokhod's representation theorem allows us to assume that this convergence holds almost surely.  The graph distance between arbitrary vertices $v$ and $v'$ in the tree $\mB_n$ is given by
	\begin{align}
		\label{eq:lca}
		d_{\mB_n}(v, v') = h_{\mB_n}(v) + h_{\mB_n}(v') - 2 h_{\mB_n}(\mathrm{lca}(v, v')),
	\end{align}
	with $\mathrm{lca}(v, v')$ denoting the lowest common ancestor of $v$ and $v$, and $h_{\mB_n}(\cdot)$ is referring to the height of a vertex. 	Since the Brownian excursion is continuous,  Equation~\eqref{eq:lca} and standard arguments imply that almost surely
	\begin{align}
		\label{eq:comeoncomeon}
		n^{-1/2} \sup_{v,v'} d_{\mB_n}(v,v')  \to 0
	\end{align}
	with the indices $v$ and $v'$ ranging over all vertices in $\mB_n$ whose positions in the depth-first-search order differ by at most $n^{3/4}$.
	
	Using Equation~\eqref{eq:bcc} and the Borel--Cantelli lemma, it follows from~\eqref{eq:comeoncomeon} that almost surely $\mT_n$ has the property
	\begin{align}
		n^{-1/2} \sup_{v_n,v_n'} d_{\mB_n}(v_n,v'_n)  \to 0.
	\end{align}
	By~\cite[Cor. 7.5.2]{MR3024835}, this implies that the Prokhorov distance between $\mu_{\mB_n}$ and $\mu_{\mT_n}$ converges almost surely to zero after rescaling distances by $\kappa_{\mathrm{tree}} n^{-1/2}$. This completes the proof.
\end{proof}

\subsection{The size-biased tree}

We construct the tree $\mT^{(\ell)}$ with spine length $\ell$ in the following way.
Let $u_0, u_1, \dots, u_{\ell}$ denote a sequence of black vertices, which form the spine
of $\mT^{(\ell)}$. The tip $u_{\ell}$ of the spine receives offspring according to
$(\xi^\circ, \zeta^\circ)$ and a uniformly selected white child is marked. For each 
$0\leq i<\ell$ the vertex $u_i$ receives offspring according to an independent copy of
$(\xi^\bullet, \zeta^\bullet)$ and we identify $u_{i+1}$ with a uniformly selected black child.
The construction of $\mT^{(\ell)}$ is finalized by identifying each black non-spine vertex 
with the root of an independent copy of the unconditioned tree $\mT$.

Again, a decorated tree $(\mT^{(\ell)}, \alpha^{(\ell)})$ can be constructed by assigning
random decorations $\alpha^{(\ell)}(v)$ to each $v\in\mT^{(\ell)}$ in the canonical way.

\begin{lemma}\label{lemma:size_biased}
	For any finite marked tree $(\tau, \alpha_\tau) \in \mathfrak{A}_f$ with the height of the marked vertex equal to $\ell$ it holds that
	\[
		\PP((\mT^{(\ell)}, \alpha^{(\ell)}) = (\tau, \alpha_\tau)) =  \PP( (\mT, \alpha) = (\tilde{\tau}, \alpha_\tau)  ) / \ex[\zeta],
	\]
	with $\tilde{\tau}$ denoting the unmarked tree obtained from $\tau$ by forgetting which vertex is marked.
\end{lemma}
\begin{proof}
    For any $\tau\in\mathfrak{X}_f$ with a marked
    white vertex at height $\ell+1$, denote by $v_{\ell}$ the tip of the spine (the parent of
    the marked white vertex) and by  $v_0, \dots, v_{\ell-1}$ the remaining vertices of the
    spine (the ancestors of $v_{\ell}$). Denote also the black and white outdegrees of its vertices
    by $d^{\bullet}_{\tau}(v)$ and $d^{\circ}_{\tau}(v)$, respectively.
    Consider first the undecorated setting. We have that
    \begin{align*}
        \PP(\mT^{(\ell)} = \tau) &= \left( \prod_{v \in \tau\setminus\{v_0, \dots, v_{\ell}\}} 
            \PP((\xi, \zeta) =  (d^{\bullet}_{\tau}(v), d^{\circ}_{\tau}(v))) \right) \\
            &\qquad \cdot \left( \prod_{v \in \{v_0,\dots, v_{\ell-1}\}} \PP((\xi^{\bullet}, 
            \zeta^{\bullet}) =  (d^{\bullet}_{\tau}(v), d^{\circ}_{\tau}(v)))\frac{1}{d^{\bullet}_{\tau}(v)} \right) \\
            &\qquad \cdot \PP((\xi^{\circ}, \zeta^{\circ}) =  (d^{\bullet}_{\tau}(v_{\ell}), d^{\circ}_{\tau}(v_{\ell})))
            \frac{1}{d^{\circ}_{\tau}(v_{\ell})}\\
        &= \left( \prod_{v \in \tau\setminus\{v_0, \dots, v_{\ell}\}} 
            \PP((\xi, \zeta) =  (d^{\bullet}_{\tau}(v), d^{\circ}_{\tau}(v))) \right) \\
            &\qquad \cdot \left( \prod_{v \in \{v_0, \dots, v_{\ell-1}\}} \PP((\xi, \zeta) = 
            (d^{\bullet}_{\tau}(v), d^{\circ}_{\tau}(v))) \right) \\
            &\qquad \cdot \PP((\xi, \zeta) =  (d^{\bullet}_{\tau}(v_{\ell}), d^{\circ}_{\tau}(v_{\ell})))
            \frac{1}{\ex[\zeta]}\\
        &= \left( \prod_{v \in \tau}  \PP((\xi, \zeta) =  (d^{\bullet}_{\tau}(v), 
            d^{\circ}_{\tau}(v))) \right) \frac{1}{\ex[\zeta]}\\
        &= \PP(\mT = \tilde{\tau}) \frac{1}{\ex[\zeta]}.
    \end{align*}
    Then, taking into account the decorations we conclude that
    \begin{align*}
        \PP((\mT^{(\ell)}, \alpha^{(\ell)}) = (\tau, \alpha_{\tau})) &= \PP((\mT^{(\ell)}, \alpha^{(\ell)}) = 
            (\tau, \alpha_{\tau}) \mid \mT^{(\ell)} = \tau)   \cdot \PP(\mT^{(\ell)} = \tau) \\
        &= \PP((\mT, \alpha) =  (\tilde{\tau}, \alpha_{\tau}) \mid \mT = 
            \tilde{\tau})   \cdot \PP(\mT = \tilde{\tau})/ \ex[\zeta] \\
        &= \PP((\mT, \alpha) =  (\tilde{\tau}, \alpha_{\tau}))/ \ex[\zeta].
    \end{align*}
\end{proof}

\begin{remark}
    The same construction can be extended to $\ell=\infty$, obtaining the tree 
    $\mT^{(\infty)}$. Notice that here, as opposed to  $\mT^{\circ}$, the spine is growing
    downwards and thus there is no vertex with offspring distributed as $(\xi^{\circ}, \zeta^{\circ})$
    and no marked vertex. It is easy to adapt the proof of the monotype case \cite{JANSON2012}
    to see that $\mT_n \stackrel{d}{\longrightarrow}\mT^{(\infty)}$ in the local topology
    with respect to the root.
\end{remark}

By Skorokhod's representation theorem, we may assume without loss of generality that
$\mT_n\stackrel{\text{a.s.}}{\longrightarrow}\mT^{(\infty)}$. Assigning decorations on both
sides in the canonical way immediately implies a corresponding convergence of $(\mT_n,
\alpha_n)$ towards $(\mT^{(\infty)}, \alpha^{(\infty)})$. Arguing analogously as for the 
random marked vertex case, it follows that $\mG_{k, n}^{(k)}\stackrel{d}{\longrightarrow}
\mG_{k, (\infty)}^{(k)}$, where $\mG_{k, (\infty)}^{(k)}=\Phi(\mT^{(\infty)}, \alpha^{(\infty)})$,
in the local topology with respect to, say, the first vertex in the root $k$-clique.

\subsection{Comparison of different metrics}

\begin{lemma}\label{lem:white_bound}
    The maximum white outdegree in $\mT_n$ is $\mathcal{O}_p(\log n)$.
\end{lemma}
\begin{proof}
    Let $C>0$ be a large enough constant. Then, using Proposition \ref{pro:offspring}
    \begin{align*}
        \PP(\exists& v\in\mT_n \text{ with } d^{\circ}(v) > C\log n )\\
        &= \PP(\exists v\in\mT \text{ with } d^{\circ}(v) > C\log n \mid \#_2\mT = n)\\
        &= \frac{\PP(\exists v\in\mT \text{ with } d^{\circ}(v) > C\log n, \#_2\mT = n)}{\PP(\#_2\mT =n)}\\
        &=\sum_{1\leq m \leq n}c_k^{-1} n^{3/2} \PP(\exists v\in\mT \text{ with } d^{\circ}(v) >
        C\log n, \#_1\mT = m, \#_2\mT = n),
    \end{align*}
    where $c_k = \frac{\boldsymbol \alpha_k k! \rho_k^{-k} }{G_k^{(k)}(\rho_k)}$.
    For each of these events, there is at least one out of $m$ independent copies of
    $(\xi, \zeta)$ with $\zeta>C\log n$. But by Proposition \ref{pro:offspring} $(\xi, \zeta)$ has 
    finite exponential moments, so
    \[
        \PP(\zeta>C\log n) < a e ^{-b \cdot C\log n} = a n^{-bC},
    \]
    for some constants $a, b > 0$. Therefore,
    \begin{align*}
        \PP(\exists v\in\mT_n \text{ with } d^{\circ}(v) > C\log n )  &\leq\sum_{1\leq m \leq n}c_k^{-1} n^{3/2} (1-(1-a n^{-bC})^{m})\\
        &\leq n c_k^{-1} n^{3/2} (1-(1-a n^{-bC})^{n})\\
        &\leq c_k^{-1} n^{5/2} (1-(1-n \cdot a n^{-bC}))\\
        &= c_k^{-1} a  n^{-bC + 7/2}
    \end{align*}
    and the result follows.
\end{proof}

In the rest of this section we assume any tree $(\mT, \alpha)$ and its corresponding graph
$\Phi((\mT, \alpha))$ to be coupled. A black vertex and its corresponding clique
receive the same name and the same goes for white vertices and their corresponding
vertices. For every white vertex $v$ of $(\mT, \alpha)$, define
\begin{align*}
    h_{\mT_n}(v) & := \text{the height of } v \text{ in } \mT_n, \\
    h_{\Phi((\mT, \alpha))}(v) & := \text{the graph distance in } \Phi((\mT, \alpha))
        \text{ from the first vertex in}\\
        &\hphantom{{}:={}}\text{the root clique to }v.
\end{align*}
For every black vertex $v$ of $(\mT, \alpha)$, define
\begin{align*}
    h_{\mT_n}(v) & := \text{the height of } v \text{ in } \mT_n \\
    h_{\Phi((\mT, \alpha))}(v) & := \text{the shortest graph distance in } \Phi((\mT, \alpha))
        \text{ from the first}\\
        &\hphantom{{}:={}}\text{vertex in the root clique to any vertex in
        the clique } v.
\end{align*}

Our goal is to show that if a vertex $v$ of $(\mT_n, \alpha_n)$ has large enough height
$h_{\mT_n}(v)$, then $h_{\mG_{k, n}^{(k)}}(v)$ concentrates around $\gamma_k h_{\mT_n}(v)$,
as $n\to\infty$, for some constant $\gamma_k>0$. In order to do so, we will use the size-biased tree
$\mT^{(\ell)}$ constructed in the previous section. Recall that $\mT^{(\ell)}$
has a spine $u_0, u_1, \dots, u_{\ell}$ and that all vertices in the spine except
for its tip $u_{\ell}$ receive offspring according to an independent copy of
$(\xi^\bullet, \zeta^\bullet)$. The tree is then decorated to obtain
$(\mT_n, \alpha_n)$ by chosing uniformly for every vertex a decoration that is compatible
with its offspring. Consider, for $0\leq i\leq \ell$, the random variables
\[
    S_i:= h_{\mG_{k, n}^{(k)}}(u_i).
\]
Consider also the random variables $X_i\in\{1, \cdots, k\}$ defined as
the number of vertices in the clique $u_i$ at distance $S_i$ from the 
first vertex in the root clique of $\mG_{k, n}$.
Since any two vertices in a clique are adjacent, the graph distance in
$\mG_{k, n}$ from the first vertex in the root clique to any of the remaining
$k-X_i$ vertices in the clique $u_i$ is $S_i+1$.  By definition,
$X_0 = 1$ and $S_0=0$. Now, observe that
\begin{itemize}
    \item The random pair $(X_i, S_i)$ depends only on the previous
    pair\\ $(X_{i-1}, S_{i-1})$.
    \item The random pair $(X_i, S_i-S_{i-1})$ depends only on $X_{i-1}$.
\end{itemize}
This means that $(X_i, S_i)$ is a discrete Markov additive process. Let 
us also verify that it is irreducible and aperiodic. 

Suppose that one of the $(k+1)$-connected components in the
offspring of $u_i$ is a $(k+1)$-clique and that $u_{i+1}$ is one of its
cliques (different from $u_i$).
Then, the clique $u_{i+1}$ is obtained from $u_i$ by replacing one of its 
vertices by a new vertex at distance $S_i+1$ from the first vertex in the root
clique. We can thus say the following about $X_{i+1}$ depending on $X_i$:
\begin{itemize}
    \item If $X_i=k$, then $X_{i+1}=k-1$.
    \item If $1<X_i<k$, then $X_{i+1}$ is either equal to $X_i$ or $X_i-1$,
    depending on whether the replaced vertex was at distance $S_i+1$ or
    $S_i$ from the first vertex in the root clique.
    \item If $X_i=1$, then $X_{i+1}$ is either equal to $1$ or $k$,
    depending on the same.
\end{itemize}
Since all states are reachable from any state (possibly in many steps)
and some states are reachable from themselves, then the process is 
irreducible and aperiodic. This together with the fact that the 
offspring distribution has finite exponential moments allows us to use the
large deviation results in \cite{INN1985} for the additive component $S_i$.

\begin{lemma}\label{lem:1st_concentration}
    There exists a constant $\gamma_k>0$ such that, for every $\varepsilon>0$, if $E$ 
    denotes the event that there exists a  vertex $v\in(\mT_n, \alpha_n)$ with
    $h_{\mT_n}(v)\geq\log ^3 n $ and 
    $h_{\mG_{k, n}^{(k)}}(v) \notin \gamma_k (1\pm \varepsilon) h_{\mT_n}(v)$, then 
    $
    \PP(E) 
    \le n^{-\Theta(\log^2 n)}.
    $
\end{lemma}
\begin{proof}
    Let $\mathcal{E}_{\ell, n}$ be the set of decorated trees $(\tau, \alpha_{\tau})\in\mathfrak{A}$
    with $n$ white vertices, a white marked vertex $v_{\tau}$ of height $\ell$, and such that
    the distance in $\Phi(\tau, \alpha_{\tau})$ between the first vertex in the root $k$-clique 
    and $v_{\tau}$ is not in the range $ \gamma_k (1\pm \varepsilon)\ell$. Then, using
    Proposition \ref{pro:offspring} and Lemma \ref{lemma:size_biased},
    \begin{align}
    	\label{eq:idarg}
        \PP(E)&\leq \sum_{\ell=\lfloor\log^3n\rfloor}^n\sum_{(\tau, \alpha_{\tau})\in
            \mathcal{E}_{\ell, n}}\PP((\mT_n, \alpha_n)=(\tilde{\tau}, \alpha_{\tau})) \nonumber\\
        &= \sum_{\ell=\lfloor\log^3n\rfloor}^n\sum_{(\tau, \alpha_{\tau})\in
            \mathcal{E}_{\ell, n}}\PP((\mT, \alpha)=(\tilde{\tau}, \alpha_{\tau})\mid \#_2\mT =n) \nonumber\\
        &= \sum_{\ell=\lfloor\log^3n\rfloor}^n\sum_{(\tau, \alpha_{\tau})\in
            \mathcal{E}_{\ell, n}}\frac{\PP((\mT, \alpha)=(\tilde{\tau}, \alpha_{\tau}))}{\PP(\#_2\mT =n)} \nonumber\\
        &=   c_k^{-1} n^{3/2}\ex[\zeta]\sum_{\ell=\lfloor\log^3n\rfloor}^n
            \sum_{(\tau, \alpha_{\tau})\in \mathcal{E}_{\ell, n}}
            \PP((\mT^{(\ell)}, \alpha^{(\ell)}) =(\tau, \alpha_{\tau})) \nonumber\\
        &=   c_k^{-1} n^{3/2}\ex[\zeta]\sum_{\ell=\lfloor\log^3n\rfloor}^n
            \PP((\mT^{(\ell)}, \alpha^{(\ell)}) \in \mathcal{E}_{\ell, n}), 
    \end{align}
    where $c_k = \frac{\boldsymbol \alpha_k k! \rho_k^{-k} }{G_k^{(k)}(\rho_k)}$.
    But the distance in $\Phi((\mT^{(\ell)}, \alpha^{(\ell)}))$ between the first vertex in the
    root clique and the marked vertex, call it $d$, satisfies
    \[
        d=S_{\ell} + \mathcal{O}_p(\log n).
    \]
    Indeed, the distance to clique corresponding to the tip of the  spine is given by the 
    discrete Markov additive process $(X_i, S_i)$ described above and the  remaining 
    distance is $\mathcal{O}_p(\log n)$ because we know from Lemma \ref{lem:white_bound}
    that the size of a $(k+1)$-connected component is $\mathcal{O}_p(\log n)$. It thus
    follows from \cite[Thm. 5.1]{INN1985} (see also \cite[Rem. 3.5 and Sec. 7, (ii)]{INN1985}) that, for large enough $\ell$,
    \[
        \PP((\mT^{(\ell)}, \alpha^{(\ell)}) \in \mathcal{E}_{\ell, n}) < ae^{-b\ell},
    \]
    for some $a, b> 0$. Therefore, for large enough $n$,
    \begin{align*}
        \PP(E)&< c_k^{-1} n^{3/2}\ex[\zeta]\sum_{\ell=\lfloor\log^3n\rfloor}^n
            ae^{-b\ell}\\
            &= n^{- \Theta(\log^2 n)}.
    \end{align*}
\end{proof}

\begin{lemma}\label{lem:2nd_concentration}
    For every vertex $v\in(\mT_n, \alpha_n)$, it holds that 
    \[
        h_{\mG_{k, n}^{(k)}}(v) \in\gamma_k (1\pm \varepsilon) h_{\mT_n}(v) + \mathcal{O}_p(\log^4 n).
    \]
\end{lemma}
\begin{proof}
    Lemma \ref{lem:1st_concentration} covers the case where $h_{\mT_n}(v)\geq\log ^3 n$.
    Suppose now that $ h_{\mT_n}(v) <\log ^3 n $. Then, the distance  $h_{\mG_{k, n}^{(k)}}(v)$
    is bounded by $h_{\mT_n}(v)$ times the size of each $(k+1)$-connected component,
    which is bounded by  $\mathcal{O}_p(\log n)$ by Lemma~\ref{lem:white_bound}, and so the result follows.
\end{proof}

We are now ready to prove Theorem \ref{te:scaling}.

\begin{proof}[Proof of Theorem~\ref{te:scaling}]
    For any two white vertices $u, v\in(\mT_n, \alpha_n)$ it is clear that
    \[
        d_{\mT_n}(u, v) = h_{\mT_n}(u)+h_{\mT_n}(u)-2h_{\mT_n}(w),
    \]
    where $w$ is the lowest common ancestor of $u$ and $v$ in $\mT_n$.
    And it also holds that
    \[
        d_{\mG_{k, n}^{(k)}}(u, v) = h_{\mG_{k, n}^{(k)}}(u)+h_{\mG_{k, n}^{(k)}}(v)-2h_{\mG_{k, n}^{(k)}}(w) + \mathcal{O}_p(\log n),
    \]
    since the shortest path in $\mG_{k, n}^{(k)}$ between $u$ and $v$ may not go
    through the $k$-clique $w$ but it certainly goes through
    some of the white children of $w$.
    
    Therefore, using Lemma \ref{lem:2nd_concentration} we obtain that, for every $\varepsilon>0$,
    \[
        d_{\mG_{k, n}^{(k)}}(u, v) \in \gamma_k (1\pm \varepsilon)d_{\mT_n}(u, v) + \mathcal{O}_p(\log^4 n).
    \]
    And thus it follows that
    \[
        \frac{\abs{d_{\mG_{k, n}^{(k)}}(u, v) - \gamma_k d_{\mT_n}(u, v)}}{\sqrt{n}}
        \stackrel{p}{\longrightarrow} 0.
    \]
    This together with Proposition \ref{pro:scaling limit} implies that
    \[
        \left( \mG_{k,n}^{(k)}, \kappa_{k}n^{-1/2}d_{\mG_{k,n}^{(k)}}, \mu_{\mG_{k,n}^{(k)}} \right) \stackrel{d}{\longrightarrow}
        \left( \mathscr{T}_e, d_{\mathscr{T}_e}, \mu_{\mathscr{T}_e} \right)
    \]
    in the Gromov--Hausdorff--Prokhorov sense as $n \to \infty$, where $\kappa_k
    =\gamma_k^{-1}$. Finally,
    the result follows from this by Lemma \ref{le:deroot}.
\end{proof}

Finally, we prove the tail-bounds for the diameter.

\begin{proof}[Proof of Theorem~\ref{te:tailbound}]
Any graph from $\mathcal{G}_{k, n}$ has at least $k(n-k)+1$ many $k$-cliques. Indeed, take a perfect elimination ordering of the vertices in the graph and remove them one by one. Since a vertex in a minimal separator will never be removed, the graph remains $k$-connected until there is only a $k$-clique left. Therefore, every time a vertex is removed, the number of $k$-cliques decreases
by at least $k$, and the claim follows.

Combining this fact with the asymptotic~\ref{eq:asymptoticclique}, it follows that it suffices to verify the stated tail-bounds for the diameter of $\mG_{k,n}^{(k)}$ instead of $\mG_{k,n}$. Furthermore, since the diameter is at most twice the height plus 1 (with height referring to the maximal distance of a vertex from the root $k$-clique), it hence suffices to show such a bound for the height of~$\mG_{k,n}^{(k)}$. Any vertex $v$ in $\mG_{k,n}^{(k)}$ that does not belong to the root-clique corresponds to a white vertex (also denoted by $v$) in the tree $\mT_n$. Similar as argued in the proof of Theorem~\ref{te:scaling}, the geodesics in $\mG_{k,n}^{(k)}$ from the root-clique to $v$ need to pass through the cliques corresponding to sequence $v_0, \ldots, v_\ell$ of black ancestors of $v$ in $\mT_n$. Therefore, apart from the first vertex that belongs to the root-clique all vertices in such a geodesic correspond to white children of $v_0, \ldots, v_\ell$ in $\mT_n$. Thus, the sum $S(v)$ of the number of white children of $v_0, \ldots, v_\ell$ is an upper bound for the height of $v$ in $\mG_{k,n}^{(k)}$. Therefore, in order to prove Theorem~\ref{te:tailbound}, it suffices to show that there exist constants $C,c>0$ such that
\begin{align}
	\label{eq:fin3}
	\mathbb{P}(\max_v S(v) \ge x) \le C \exp(-cx^2 / n)
\end{align}
for all  $n$ and $x>0$, with the index $v$ ranging over the white vertices in the tree $\mT_n$. Since we may always replace $C$ by some larger constant and $c$ by some smaller constant it furthermore suffice to verify such a bound for all $x> \sqrt{n}$ instead of $x>0$. Furthermore, since the height can be at most $n$, it also suffices to consider $x \le n$.

With foresight, set \[
\alpha = \frac{1}{2 \mathbb{E}[\zeta^\bullet]}.
\] Proposition~\ref{pro:scaling limit} ensures that there exist constants $C_1, c_1>0$ with
\begin{align}
	\label{eq:parta}
	\PP(\mathrm{H}(\mT_n) \ge \alpha x) \le C_1 \exp(-c_1 x^2 / n).	
\end{align}
By an identical calculation as in Equation~\eqref{eq:idarg} the probability that $\mathrm{H}(\mT_n) < \alpha x$ and at the same time there exists of some white vertex $v$ with $S(v)>x$ is bounded by
\[
O(n^{3/2}) \sum_{\ell = 1}^{\lfloor \alpha x \rfloor} p_\ell,
\]
with $p_\ell$ denoting the probability that the sum of white children of the black ancestors of the marked vertex of $\mT^{(\ell)}$ is larger than $x$. These numbers of children are distributed like independent copies $\zeta_1^\bullet, \ldots, \zeta_\ell^\bullet$ of $\zeta^\bullet$. Hence
\[
	p_\ell= \PP\left(\sum_{i=1}^\ell \zeta_i^\bullet > x\right).
\]
Using $\ell \le \alpha x$ and $x \ge \sqrt{n}$ it follows by Lemma~\ref{le:deviation} that
\[
	p_\ell \le \exp(- \Theta(x)).
\]
This means our upper bound simplifies to
\begin{align*}
O(n^{3/2}) \sum_{\ell = 1}^{\lfloor \alpha x \rfloor} p_\ell &= O(n^{3/2})x \exp(- \Theta(x)) \\
&= O(x^4) \exp(- \Theta(x)) \\
&= \exp(- \Theta(x)).
\end{align*}
Together with Inequality~\eqref{eq:parta} this proves Inequality~\eqref{eq:fin3} and hence completes the proof.

\end{proof}

\appendix

\section{Deviation inequality}

In our proof we make  use of the following medium deviation inequality found in most textbooks on the subject. See for example~\cite[Ex. 1.4]{MR3309619}.

\begin{lemma}
	\label{le:deviation}
	Let $(X_i)_{i \in \mathbb{N}}$ be an i.i.d. family of real-valued random variables with $\mathbb{E}[X_1] = 0$ and $\mathbb{E}[e^{t X_1}] < \infty$ for all $t$ in some open interval around zero. Then there are constants $\delta, c>0$ such that for all $n\in \mathbb{N}$, $x \ge 0$ and $0 \le\lambda\le\delta$ it holds that \[\PP(|X_1 + \ldots + X_n| \ge x) \le 2 \exp(c n \lambda^2 - \lambda x).\]
\end{lemma}


\section*{Acknowledgement}

The authors are grateful to Louigi Addario-Berry for helpful discussions on random trees with prescribed degree sequences. The authors also warmly thank the referees for their helpful remarks. The first author acknowledges support from the CERCA Programme/Generalitat de Catalunya and from the Spanish State Research Agency, through the Severo Ochoa and María de Maeztu Program for Centers and Units of Excellence in R\&D (CEX2020-001084-M). The second author is funded in part by the Austrian Science Fund (FWF) 10.55776/PAT6732623.

\bibliographystyle{abbrv}
\bibliography{main}

\end{document}